\documentclass{article}%
\usepackage{amssymb}
\usepackage{amsmath}
\usepackage{amsfonts}
\usepackage{graphicx}%
\newtheorem{theorem}{Theorem} [section]

\newtheorem{definition}[theorem]{Definition}
\newtheorem{example}[theorem]{Example}

\newtheorem{proposition}[theorem]{Proposition}

\newenvironment{proof}[1][Proof]{\noindent\textbf{#1.} }{\ \rule{0.5em}{0.5em}}
\begin{document}

\author{Yulia Kempner$^{1}$and Vadim E. Levit$^{1,2}$\\$^{1}
$Dept. of Computer Science \\\ Holon Institute of Technology \\\ yuliak@hit.ac.il\\$^{2}
$Dept. of Computer Science and Mathematics\\\ Ariel University Center of Samaria\\\ levitv@ariel.ac.il}
\title{Duality between quasi-concave functions and monotone linkage functions}
\date {}
\maketitle

\begin{abstract}
A function $F$ defined on all subsets of a finite ground set $E$ is
\textit{quasi-concave} if $F(X\cup Y)\geq\min\{F(X),F(Y)\}$ for all
$X,Y\subset E$. Quasi-concave functions arise in many fields of mathematics
and computer science such as social choice, theory of graph, data mining,
clustering and other fields.

The maximization of quasi-concave function takes, in general,
exponential time. However, if a quasi-concave function is defined by
associated monotone linkage function then it can be optimized by the
greedy type algorithm in a polynomial time.

Quasi-concave functions defined as minimum values of monotone
linkage functions were considered on antimatroids, where the correspondence
between quasi-concave and bottleneck functions was shown \cite{EJC}. The goal
of this paper is to analyze quasi-concave functions on different families of
sets and to investigate their relationships with monotone linkage functions.

\end{abstract}

\section{Preliminaries}

Many combinatorial optimization problems can be formulated as: for a given
\textit{set system }over $E$ (i.e., for a pair $(E,\mathcal{F})$ where
$\mathcal{F}\subseteq2^{E}$ is a family of feasible subsets of finite set
$E$), and for a given function $F:\mathcal{F}\rightarrow\mathbf{R}$, find an
element of $\mathcal{F}$ for which the value of the function $F$ is extremal.
In general, this optimization problem is \textbf{NP}-hard, but for some
specific functions and set systems the problem may be solved in polynomial
time. For instance, modular cost functions can be optimized over matroids by
greedy algorithms \cite{Edmonds}, and bottleneck functions can be maximized
over greedoids \cite{Greedoids}. Another example is about set functions
defined as minimum values of monotone linkage functions. These functions are
known as quasi-concave set functions. Such set functions can be maximized by a
greedy type algorithm over the family of all subsets of $E$ \cite{KMM}%
,\cite{Mirkin},\cite{Mullat},\cite{Zaks}, over antimatroids and convex
geometries \cite{EJC},\cite{Malta},\cite{Mirkin2}, join-semilattices
\cite{MuchnikShvar} and meet-semilattices \cite{DAM}.

Originally \cite{Malishevski}, these functions were defined on the Boolean
$2^{E}$:
\begin{equation}
\text{for\ each}\ X,Y\subset E,\ F(X\cup Y)\geq\min\{F(X),F(Y)\}.
\label{Mal_def}%
\end{equation}

In this work we extend this definition to set systems that are not necessarily
closed under union.

Let $E$ be a finite set, and a pair $(E,\mathcal{F})$ be a set
system\textit{\ }over $E$.

\begin{definition}
A minimal feasible subset of $E$ that includes a set $X$ is called a
\textit{cover of }$X$.
\end{definition}

\smallskip We will denote by \-$\mathcal{C}(X)$ the family of covers of $X$.

\begin{definition}
A function $F$ defined on a set system $(E,\mathcal{F})$ is quasi-concave if
for\ each\ $X,Y\in\mathcal{F}$,\ and $Z\in\mathcal{C}(X\cup Y)$,
\begin{equation}
F(Z)\geq\min\{F(X),F(Y)\}. \label{q_con}%
\end{equation}

\end{definition}

If a set system is closed under union, then the family of covers
$\mathcal{C}(X\cup Y)$ contains the unique set $X\cup Y$, and the inequality
(\ref{q_con}) coincides with the original inequality (\ref{Mal_def}).

Here we give definitions of some set properties that are discussed in the
following section. We will use $X\cup x$ for $X\cup\{x\}$, and $X-x$ for
$X-\{x\}$.

\begin{definition}
A non-empty set system $(E,\mathcal{F})$ is called \textit{accessible }if for
each non-empty $X\in\mathcal{F}$, there is an $x\in X$ such that
$X-x\in\mathcal{F}$.
\end{definition}

For each non-empty set system $(E,\mathcal{F})$ accessibility implies that
$\varnothing\in\mathcal{F}$.

\begin{definition}
A \textit{closure operator }, $\tau:2^{E}\rightarrow2^{E}$, is a map
satisfying the closure axioms:

C1: $X\subseteq\tau(X)$

C2: $X\subseteq Y\Rightarrow\tau(X)\subseteq\tau(Y)$

C3: $\tau(\tau(X))=\tau(X)$.
\end{definition}

\begin{definition}
The set system $(E,\mathcal{F})$ is a \textit{closure space} if it satisfies
the following properties

(1) $\varnothing\in\mathcal{F}$, $E\in\mathcal{F}$

(2) $X,Y\in\mathcal{F}$ implies $X\cap Y\in\mathcal{F}$.
\end{definition}

Let a set system $(E,\mathcal{F})$ be a closure space, then the operator
\begin{equation}
\tau(A)=\cap\{X:A\subseteq X\text{ }and\text{ }X\in\mathcal{F}\}
\label{cl_oper}%
\end{equation}
is a closure operator.

Convex geometries were introduced by Edelman and Jamison \cite{Edelman} as a
combinatorial abstraction of "convexity".

\begin{definition}
\cite{Greedoids} The closure space $(E,\mathcal{F})$ is a \textit{convex
geometry} if the family $\mathcal{F}$ satisfies the following property
\begin{equation}
X\in\mathcal{F}-E\ implies\ X\cup x\in\mathcal{F}\ for\ some\ x\in
E-X.\label{up_chain}%
\end{equation}

\end{definition}

It is easy to see that property (\ref{up_chain}) is dual to accessibility.
Then, we will call it \textit{up-accessibility}. If in each non-empty
accessible set system one can reach the empty set $\varnothing$ from any
feasible set $X\in\mathcal{F}$ by moving down, so in each non-empty
up-accessible set system $(E,\mathcal{F})$ the set $E$ may be reached by
moving up.

It is clear that a complement set system $(E,\overline{\mathcal{F}}\ )$
(system of complements), where $\overline{\mathcal{F}}=\{X\subseteq
E:E-X\subseteq\mathcal{F}\}$, is up-accessible\textit{\ }if and only if the
set system $(E,\mathcal{F})$ is accessible.

In fact, accessibility means that for all $X\in\mathcal{F}$ there exists a
chain $\varnothing=X_{0}\subset X_{1}\subset...\subset X_{k}=X$ such that
$X_{i}=X_{i-1}\cup x_{i}$, $x_{i}\in X_{i}-X_{i-1}$ and $X_{i}\in\mathcal{F}$
for $0\leq i\leq k$, and up-accessibility implies the existence of the
corresponding chain $X=X_{0}\subset X_{1}\subset...\subset X_{k}=E$. Consider
a set family for which this \textit{chain property} holds for each pair of
sets $X\subset Y$.

\begin{definition}
A set system $(E,\mathcal{F})$ satisfies the \textit{chain property }if for
all $X,Y\in\mathcal{F}$, and $X\subset Y$, there exists an $y\in Y-X$ such
that $Y-y\in\mathcal{F}$. We call the system a chain system.
\end{definition}

In other words, a set system $(E,\mathcal{F})$ satisfies the chain
property\textit{\ }if for all $X,Y\in\mathcal{F}$, and $X\subset Y$, there
exists an $y\in Y-X$ such that $X\cup y\in\mathcal{F}$.

\begin{proposition}
$(E,\overline{\mathcal{F}}\ )$ is a chain system if and only if
$(E,\mathcal{F})$ is a chain system as well.
\end{proposition}

\begin{proof}
Let $X,Y\in\overline{\mathcal{F}}$, and $X\subset Y$, then there exist
$\overline{X}=E-X$ and $\overline{Y}$ such that $\overline{Y}\subset
\overline{X}\overline{\text{,}}$ and there is $y\in\overline{X}-\overline{Y} $
such that $\overline{Y}\cup y\in\mathcal{F}$. Since $\overline{X}-\overline
{Y}=\overline{X}\cap Y=Y-X$, we have $y\in Y-X$. In addition, $\overline
{Y}\cup y\in\mathcal{F}$ implies $Y-y\in\overline{\mathcal{F}}$, that
completes the proof.
\end{proof}

Consider a relation between accessibility and the chain property. If
$\varnothing\in\mathcal{F}$, then accessibility follows from the chain
property. In general case, there are accessible set systems that do not
satisfy the chain property (for example, consider $E=\{1,2,3\}$ and
$\mathcal{F}=\{\varnothing,\{1\},\{2\},\{2,3\},\{1,2,3\}\}$) and vice versa,
it is possible to construct a set system, that satisfies the chain property
and it is not an accessible (for example, let now $\mathcal{F}%
=\{\{1\},\{3\},\{1,2\},\{2,3\},\{1,2,3\}\}$). In fact, if we have an
accessible set system satisfying the chain property, then the same system but
without the empty set (or without all subsets of cardinality less then some
$k$) is not accessible, but satisfies the chain property. The analogy
statements are correct for up-accessibility.

Examples of chain systems include convex geometries (see proposition
\ref{Pr_chain_convex}) and their complement systems called antimatroids,
hereditary systems (matroids, matchings, cliques, independent sets of a graph).

Consider another example of a chain system.

\begin{example}
For a graph $G=(V,E)$, the set system $(V,\mathcal{S})$ given by
\[
\mathcal{S}=\{A\subseteq V:(A,E(A))\ \text{is\ a\ connected\ subgraph\ of}%
\ G\},
\]
is a chain system. The example is illustrated in Figure \ref{fig}.
\end{example}

\begin{figure}[ptbh]
\setlength{\unitlength}{1.1cm}
\par
\begin{picture}(7,5)(0,-2.9)\thicklines
\multiput(0.3,0)(2,0){2}{\circle*{0.3}} \put(1.3,2){\circle*{0.3}}
\put(1.3,-2){\circle*{0.3}} \put(0,0){\makebox(0,0){1}}
\put(2.6,0){\makebox(0,0){3}} \put(1.3,2.3){\makebox(0,0){2}}
\put(1.3,-2.3){\makebox(0,0){4}} \put(0.3,0){\line(1,2){1}}
\put(0.3,0){\line(1,-2){1}} \put(1.3,2){\line(1,-2){1}}
\put(1.3,-2){\line(1,2){1}} \put(1.3,-3){\makebox(0,0){(a)}}
\multiput(4,0)(2,0){4}{\circle*{0.2}}
\multiput(4,1)(2,0){4}{\circle*{0.2}}
\multiput(4,-1)(2,0){4}{\circle*{0.2}} \put(7,2){\circle*{0.2}}
\put(7,-2){\circle*{0.2}} \multiput(4,-1)(2,0){4}{\line(0,1){2}}
\put(4,0){\line(6,1){6}} \put(4,-1){\line(6,1){6}}
\put(4,0){\line(2,-1){2}} \put(8,1){\line(2,-1){2}}
\put(4,1){\line(2,-1){4}} \put(6,1){\line(2,-1){4}}
\put(4,1){\line(3,1){3}} \put(6,1){\line(1,1){1}}
\put(8,1){\line(-1,1){1}} \put(10,1){\line(-3,1){3}}
\put(4,-1){\line(3,-1){3}} \put(6,-1){\line(1,-1){1}}
\put(8,-1){\line(-1,-1){1}} \put(10,-1){\line(-3,-1){3}}
\put(3.4,0.9){\makebox(0,0){\{1,2,3\}}}
\put(3.4,-0.1){\makebox(0,0){\{1,2\}}}
\put(3.5,-1.1){\makebox(0,0){\{1\}}} \put(7,2.3){\makebox(0,0){V}}
\put(5.4,0.9){\makebox(0,0){\{2,3,4\}}}
\put(5.4,-0.1){\makebox(0,0){\{2,3\}}}
\put(5.5,-1.1){\makebox(0,0){\{2\}}}
\put(7.4,0.9){\makebox(0,0){\{1,3,4\}}}
\put(7.4,-0.1){\makebox(0,0){\{3,4\}}}
\put(7.5,-1.1){\makebox(0,0){\{3\}}}
\put(10.6,0.9){\makebox(0,0){\{1,2,4\}}}
\put(10.5,-0.1){\makebox(0,0){\{1,4\}}}
\put(10.4,-1.1){\makebox(0,0){\{4\}}}
\put(7,-2.3){\makebox(0,0){$\varnothing $}}
\put(7,-3){\makebox(0,0){(b)}}
\end{picture}
\caption{$G=(V,E)$ (a) and a family of connected subgraphs (b).}%
\label{fig}%
\end{figure}
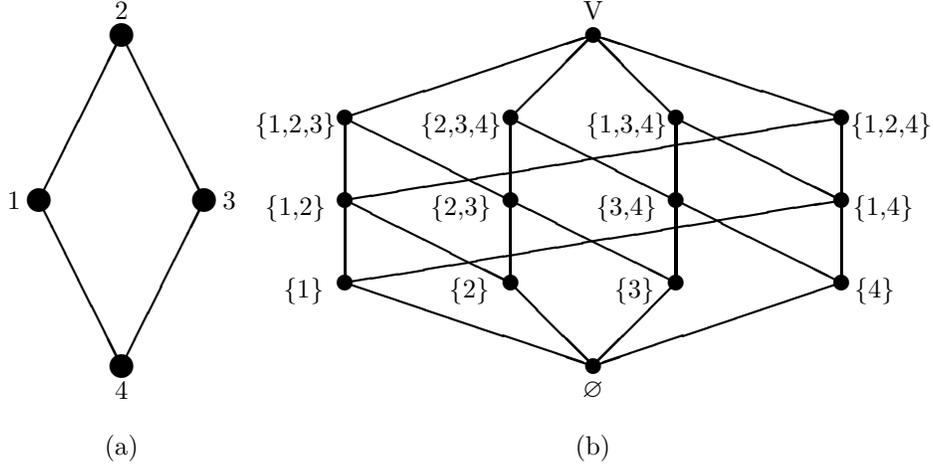

To show that $(V,\mathcal{S})$ is a chain system consider some $A,B\in
\mathcal{S}$ such that $A\subset B$. We are to prove that there exists an
$b\in B-A$ such that $A\cup b\in\mathcal{S}$. Since $B$ is a connected
subgraph, there is an edge $e=(a,b)$, where $a\in A$ and $b\in B-A$. Hence,
$A\cup b\in\mathcal{S}$.

For a set $X\in\mathcal{F}$, let $ex(X)=\{x\in X:X-x\in\mathcal{F}\}$ be the
set of \textit{extreme points }of $X$. Originally, this operator was defined
for closure spaces \cite{Edelman}. Our definition does not demand the existing
of a closure operator, but when the set system $(E,\mathcal{F})$ is a convex
geometry $ex(X)$ becomes the classical set of extreme points of a convex set
$X$.

Note, that accessibility means that for each non-empty $X\in\mathcal{F},$
$ex(X)\neq\varnothing$.

\begin{definition}
The operator $ex:\mathcal{F}\rightarrow2^{E}$ satisfies the \textit{heritage}
property if$\ X\subseteq Y$\ implies\ $ex(Y)\cap X\subseteq ex(X)$
for\ all\ $X,Y\in\mathcal{F}$.
\end{definition}

We choose the name \textit{heritage property} following B.Monjardet
\cite{Monjardet}. This condition is well-known in the theory of choice
functions where one uses also alternative terms like \textit{Chernoff
condition }\cite{Chernoff} or \textit{property }$\alpha$ \cite{Sen}. This
property is also known in the form $X-ex(X)\subseteq Y-ex(Y)$.

The heritage property means that $Y-x\in\mathcal{F}$ implies $X-x\in
\mathcal{F}$ for all $X,Y\in\mathcal{F}$ with $X\subseteq Y$ and for all $x\in
X$.

The extreme point operator of a closure space satisfies the heritage property,
but the opposite statement in not correct. Indeed, consider the following
example illustrated in Figure \ref{fig2} (a): let $E=\{1,2,3,4\}$ and
\[
\mathcal{F}=\{\varnothing
,\{1\},\{2\},\{3\},\{1,2\},\{2,4\},\{3,4\},\{1,2,3\},E\}\text{.}%
\]
It is easy to check that the extreme point operator $ex$ satisfies the
heritage property, but the set system $(E,\mathcal{F})$ is not a closure space
($\{2,4\}\cap\{3,4\}\notin\mathcal{F}$). It may be mentioned that this set
system does not satisfy the chain property. Another example (Figure \ref{fig2}
(b)) shows that the chain property is also not enough for a set system to be a
closure space. Here
\[
\mathcal{F}=\{\varnothing
,\{1\},\{4\},\{1,3\},\{3,4\},\{1,2,3\},\{2,3,4\},E\},
\]
and the constructed set system satisfies the chain property, but is not a
closure set ($\{1,3\}\cap\{3,4\}\notin\mathcal{F}$).

\begin{figure}[ptbh]
\setlength{\unitlength}{1.2cm}
\par
\begin{picture}(7,5)(0,-2.9)\thicklines
\multiput(0.5,0)(2,0){3}{\circle*{0.2}}
\multiput(0.5,-1)(2,0){3}{\circle*{0.2}} \put(2.5,-2){\circle*{0.2}}
\put(2.5,1){\circle*{0.2}} \put(2.5,2){\circle*{0.2}}
\put(0.2,0.3){\makebox(0,0){\{1,2\}}}
\put(3,0){\makebox(0,0){\{2,4\}}} \put(5,0){\makebox(0,0){\{3,4\}}}
\put(3.1,1){\makebox(0,0){\{1,2,3\}}}
\put(0.1,-1){\makebox(0,0){\{1\}}}
\put(2.9,-1){\makebox(0,0){\{2\}}} \put(5,-1){\makebox(0,0){\{3\}}}
\put(2.8,2){\makebox(0,0){E}}
\put(2.5,-2.3){\makebox(0,0){$\varnothing $}}
\multiput(0.5,-1)(2,0){3}{\line(0,1){1}} \put(0.5,0){\line(2,1){2}}
\put(0.5,0){\line(2,-1){2}} \put(0.5,-1){\line(2,-1){2}}
\put(2.5,-2){\line(0,1){1}} \put(2.5,-2){\line(2,1){2}}
\put(2.5,1){\line(0,1){1}} \put(2.5,-3){\makebox(0,0){(a)}}
\multiput(7,0)(2,0){2}{\circle*{0.2}}
\multiput(7,1)(2,0){2}{\circle*{0.2}}
\multiput(7,-1)(2,0){2}{\circle*{0.2}} \put(8,2){\circle*{0.2}}
\put(8,-2){\circle*{0.2}} \multiput(7,-1)(2,0){2}{\line(0,1){2}}
\put(7,-1){\line(1,-1){1}} \put(8,-2){\line(1,1){1}}
\put(7,1){\line(1,1){1}} \put(8,2){\line(1,-1){1}}
\put(6.4,1){\makebox(0,0){\{1,2,3\}}}
\put(6.5,0){\makebox(0,0){\{1,3\}}}
\put(6.5,-1){\makebox(0,0){\{1\}}} \put(8,2.3){\makebox(0,0){E}}
\put(9.6,1){\makebox(0,0){\{2,3,4\}}}
\put(9.5,0){\makebox(0,0){\{3,4\}}}
\put(9.5,-1){\makebox(0,0){\{4\}}}
\put(8,-2.3){\makebox(0,0){$\varnothing $}}
\put(8,-3){\makebox(0,0){(b)}}
\end{picture}
\caption{Heritage property (a) and chain property (b).}%
\label{fig2}%
\end{figure}
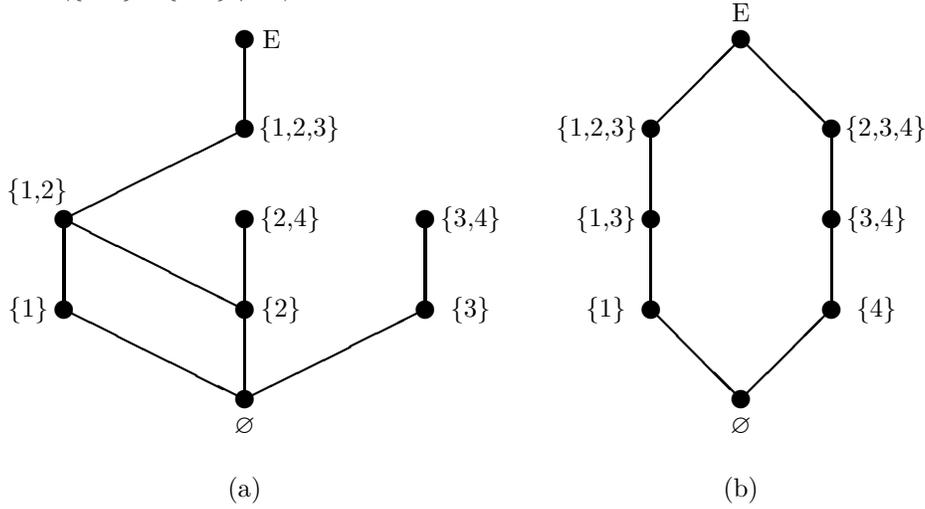

\begin{proposition}
\label{Pr_chain_convex}A set system $(E,\mathcal{F})$ is a convex geometry if
and only if

(1) $\varnothing\in\mathcal{F}$, $E\in\mathcal{F}$

(2) the set system $(E,\mathcal{F})$ satisfies the chain property

(3) the extreme point operator $ex$ satisfies the heritage property.
\end{proposition}

\begin{proof}
Let a set system $(E,\mathcal{F})$ be a convex geometry. Then the first
condition automatically follows from the convex geometry definition. Prove the
second condition. Consider $X,Y\in\mathcal{F}$, and $X\subset Y$. From
(\ref{up_chain}) follows that there is a chain
\[
X=X_{0}\subset X_{1}\subset...\subset X_{k}=E
\]
such that $X_{i}=X_{i-1}\cup x_{i\text{ }}$and
$X_{i}\in\mathcal{F}$ for $0\leq i\leq k$. Let $j$ be the least integer for
which $X_{j}\supseteq Y$. Then $X_{j-1}\nsupseteq Y$, and $x_{j}\in Y$. Thus,
$Y-x_{j}=Y\cap X_{j-1}\in\mathcal{F}$. Since $x_{j}\notin X$, the chain
property is proved. To prove that $ex(Y)\cap X\subseteq ex(X)$, consider $p\in
ex(Y)\cap X$, then $Y-p\in\mathcal{F}$ and $X\cap(Y-p)=X-p\in\mathcal{F}$,
i.e., $p\in ex(X)$.

Conversely, let us prove that the set system $(E,\mathcal{F})$ is a convex
geometry. We are to prove both up-accessibility and that $X,Y\in\mathcal{F}$
implies $X\cap Y\in\mathcal{F}$. Since $E\in\mathcal{F}$, up-accessibility
follows from the chain property.

Consider $X,Y\in\mathcal{F}$. Since $E\in\mathcal{F}$, the chain property implies that there is a chain
\[
X=X_{0}\subset X_{1}\subset...\subset X_{k}=E
\]
such that $X_{i}=X_{i-1}\cup
x_{i\text{ }}$and $X_{i}\in\mathcal{F}$ for $0\leq i\leq k$. If $j$ is the
least integer for which $X_{j}\supseteq Y$, then $X_{j-1}\nsupseteq Y$, and
$x_{j}\in Y$. Since $x_{j}\in ex(X_{j})$, we obtain $x_{j}\in ex(Y)$.
Continuing the process of clearing $Y$ from the elements that are absent in
$X$, eventually we reach the set $X\cap Y\in\mathcal{F}$.
\end{proof}

\section{Main results}

In this section we consider relationship between quasi-concave set functions
and monotone linkage functions.

Monotone linkage functions were introduced by Joseph Mullat \cite{Mullat}.

A function $\pi:E\times2^{E}\rightarrow\mathbf{R}$ is called a
\textit{monotone linkage function} if
\begin{equation}
X\subseteq Y\ \text{implies}\ \pi(x,X)\leq\pi(x,Y)\text{,\ for\ each}%
\ X,Y\subseteq E\ \text{and}\ x\in E\text{.} \label{three}%
\end{equation}

Consider function $F:(2^{E}-\{E\})\rightarrow\mathbf{R}$ defined as follows%
\begin{equation}
F(X)=\min_{x\in X}\pi(x,X)\text{.} \label{four}%
\end{equation}

\begin{example}
\label{ex1} Consider a graph $G=(V,E)$, where $V$ is a set of vertices and $E
$ is a set of edges. Let $\deg_{H}(x)$ denote the degree of vertex $x$ in the
induced subgraph $H\subseteq G$. It is easy to see that function $\pi(x,H)$ =
$\deg_{H}(x)$ is monotone linkage function and function $F(H)$ returns the
minimal degree of subgraph $H$.
\end{example}

\begin{example}
\label{ex2} Consider a proximity graph $G=(V,E,W)$, where $w_{ij}$ represents
the degree of similarity of objects $i$ and $j$. A higher value of $w_{ij}$
reflects a higher similarity of objects $i$ and $j$. Define a monotone linkage
function $\pi(i,H)=\underset{j\in H}{\sum}w_{ij}$, that measures proximity
between subset $H\subseteq V$ and their element $i$. Then the function
$F(H)=\underset{i\in H}{\min}\pi(i,H)$ can be interpreted as a measure of
density\emph{ }of set $H$.
\end{example}

It was shown \cite{Malishevski}, that for every monotone linkage function
$\pi$, function $F$ is quasi-concave on the Boolean $2^{E}$. Moreover, each
quasi-concave function may be defined by a monotone linkage function. In this
section we investigate this relation on different families of sets.

For each function $F$ defined on a set system $(E,\mathcal{F})$, we can
construct the corresponding linkage function
\begin{equation}
\pi_{F}(x,X)=\left\{
\begin{array}
[c]{ll}%
\underset{A\in\lbrack x,X]_{\mathcal{F}}}{\max}F(A), & x\in X\ \text{and}%
\ [x,X]_{\mathcal{F}}\neq\varnothing\\
\underset{A\in\mathcal{F}}{\min}F(A), & \text{otherwise}%
\end{array}
\right.  . \label{pi_F}%
\end{equation}

where $[x,X]_{\mathcal{F}}=\{A\in\mathcal{F}:x\in A\ and\ A\subseteq X\}$.

\begin{proposition}
$\pi_{F}$ is monotone.
\end{proposition}

\begin{proof}
Indeed, if $x\in X$ and $[x,X]_{\mathcal{F}}\neq\varnothing$, then $X\subseteq
Y$ implies $[x,Y]_{\mathcal{F}}\neq\varnothing$ and
\[
\pi_{F}(x,X)=\underset{A\in\lbrack x,X]_{\mathcal{F}}}{\max}F(A)\leq
\underset{A\in\lbrack x,Y]_{\mathcal{F}}}{\max}F(A)=\pi_{F}(x,Y).
\]

If $x\in X$ and $[x,X]_{\mathcal{F}}=\varnothing$, then $X\subseteq Y$ implies
$\pi_{F}(x,X)=\underset{A\in\mathcal{F}}{\min}F(A)\leq\pi_{F}(x,Y)$. It is
easy to verify the remaining cases.
\end{proof}

Let $(E,\mathcal{F})$ be an accessible set system. Denote $\mathcal{F}%
^{+}=\mathcal{F}-\varnothing$. Then, having the linkage function $\pi_{F}$, we
can construct for all $X\in\mathcal{F}^{+}$ the set function
\begin{equation}
G_{F}(X)=\underset{x\in ex(X)}{\min}\pi_{F}(x,X). \label{G_pi}%
\end{equation}

Now consider the relationship between two set functions $F$ and $G_{F}$.

\begin{proposition}
\label{PrGF}If $(E,\mathcal{F})$ is an accessible set system, then
\[
G_{F}(X)\geq F(X),\ for\ each\ X\in\mathcal{F}^{+}.
\]

\end{proposition}

\begin{proof}
Indeed,
\[
G_{F}(X)=\underset{x\in ex(X)}{\min}\pi_{F}(x,X)=\pi_{F}(x^{\ast}%
,X)=\underset{A\in\lbrack x,X]]_{\mathcal{F}}}{\max}F(A)\geq F(X),
\]

where $x^{\ast}\in\underset{x\in ex(X)}{\arg\min}\pi_{F}(x,X)$\footnote{$\arg
\min f(x)$ denote the set of arguments that minimize the function $f$.}.
\end{proof}

What conditions on the set system $(E,\mathcal{F})$ are to be satisfied to be
sure that $G_{F}$ coincides with $F$?

\begin{theorem}
Let $(E,\mathcal{F})$ be an accessible set system. Then for every
quasi-concave set function $F:\mathcal{F}^{+}\rightarrow\mathbf{R}$%
\[
G_{F}=F\ \text{on}\ \mathcal{F}^{+}%
\]
if and only if the set system $(E,\mathcal{F})$ satisfies the chain property.
\end{theorem}

\begin{proof}
Assume that the set system $(E,\mathcal{F})$ satisfies the chain property. For
each $X\in\mathcal{F}^{+}$%
\[
G_{F}(X)=\underset{x\in ex(X)}{\min}\pi_{F}(x,X)=\underset{x\in ex(X)}{\min
}F(A^{x})\text{,}%
\]
where $A^{x}$ is a set from $[x,X]_{\mathcal{F}}$ on which the value of the
function $F$ is maximal, i.e.,%
\[
A^{x}\in\arg\max_{A\in\lbrack x,X]_{\mathcal{F}}}F(A).
\]

Consider $Z$ that is a cover of $\underset{x\in ex(X)}{\cup}A^{x}$, i.e.
$Z\in\mathcal{C}(\underset{x\in ex(X)}{\cup}A^{x})$. From quasi-concavity
(\ref{q_con}) it follows that $\underset{x\in ex(X)}{\min}F(A^{x})\leq F(Z)$.
So, $G_{F}(X)\leq F(Z)$ for each $Z\in\mathcal{C}(\underset{x\in ex(X)}{\cup
}A^{x})$. Now, to prove that $G_{F}=F$, it is enough to show that
$X\in\mathcal{C}(\underset{x\in ex(X)}{\cup}A^{x})$.

In fact, the stronger proposition is correct. If $(E,\mathcal{F})$ is an
accessible chain system, then for all $X\in\mathcal{F}$ and $B^{x}\in\lbrack
x,X]_{\mathcal{F}}$%
\begin{equation}
X\in\mathcal{C}(\underset{x\in ex(X)}{\cup}B^{x})\text{.} \label{GF_F}%
\end{equation}

For each $x\in ex(X)$, $X\supseteq B^{x}$, and then $X\supseteq\underset{x\in
ex(X)}{\cup}B^{x}$. Assume, that $X$ is not a cover of $\underset{x\in
ex(X)}{\cup}B^{x}$, i.e., there is a set $Y$, such that $Y\in\mathcal{C}%
(\underset{x\in ex(X)}{\cup}B^{x})$ and $X\supset Y$. Then from the chain
property it follows that there exists an element $y\in X-Y$ such that
$X-y\in\mathcal{F}$, i.e., there exists $y\in ex(X)$ and $y\notin Y$. On the
other hand,
\[
Y\in\mathcal{C}(\underset{x\in ex(X)}{\cup}B^{x})\Rightarrow Y\supseteq
\underset{x\in ex(X)}{\cup}B^{x}\supseteq ex(X)\text{,}%
\]
contradiction that proves (\ref{GF_F}). Therefore, $G_{F}(X)\leq F(X)$, and,
with (\ref{PrGF}), $F=G$.

Conversely, assume that the set system $(E,\mathcal{F})$ does not satisfy the
chain property. Since the set system $(E,\mathcal{F})$ is an accessible
system, it means that there exist $A,B\in\mathcal{F}$ such that $A\subset B$,
$A\neq\varnothing$ and there is not any $b\in B-A$ such that $B-b\in
\mathcal{F}$, i.e., $ex(B)\subseteq A$.

It is easy to see that the function%
\[
F(X)=\left\{
\begin{array}
[c]{ll}%
1, & X=A\\
0, & \text{otherwise}%
\end{array}
\right.  .
\]
is quasi-concave.

Consider the linkage function $\pi_{F}$. Since $x\in ex(B)$ implies $x\in A$,
then
\[
\pi_{F}(x,B)=\underset{X\in\lbrack x,B]_{\mathcal{F}}}{\max}F(X)=F(A)=1
\]
Thus, $G_{F}(B)=1$, i.e. $G_{F}\neq F$.
\end{proof}

Thus, we proved that on an accessible set system satisfying the chain property
each quasi-concave function $F$ determines a monotone linkage function
$\pi_{F}$, and a set function defined as a minimum of this monotone linkage
function $\pi_{F}$ coincides with the original function $F$.

As examples of such set system may be considered greedoids \cite{Greedoids}
that include matroids and antimatroids, and antigreedoids including convex
geometries. By an antigreedoid we mean a set system $(E,\mathcal{F})$ such
that the complementary set system $(E,\overline{\mathcal{F}})$ is a greedoid.

Note, that if $F$ is not quasi-concave, the function $G_{F}$ does not
necessarily equal $F$. For example, let $\mathcal{F}=\{\varnothing
,\{1\},\{2\},\{1,2\}\}$ and let
\[
F(X)=\left\{
\begin{array}
[c]{ll}%
0, & X=\{1,2\}\\
1, & \text{otherwise}%
\end{array}
\right.
\]

Function $F$ is not quasi-concave, since $F(\{1\}\cup\{2\})<\min
(F\{1\},F\{2\})$. It is easy to check that here $G_{F}\neq F$, because
$\pi_{F}(1,\{1,2\})=\pi_{F}(2,\{1,2\})=1$, and so $G_{F}(\{1,2\})=1$.
Moreover, the function $G_{F}$ is quasi-concave. To understand this
phenomenon, consider the opposite process.

Let $(E,\mathcal{F})$ be an accessible set system. We can construct the set
function $F_{\pi}:\mathcal{F}^{+}\rightarrow\mathbf{R}$ :
\begin{equation}
F_{\pi}(X)=\underset{x\in ex(X)}{\min}\pi(x,X), \label{F_def}%
\end{equation}

based on the monotone linkage function $\pi$ defined on $E\times2^{E}$.

To extend this function to the whole set system $(E,\mathcal{F})$ define
\[
F_{\pi}(\varnothing)=\underset{(x,X)}{\min}\pi(x,X).
\]

\begin{theorem}
\label{T-2}\smallskip Let $(E,\mathcal{F})$ be an accessible set system. Then
the following statements are equivalent

$(i)$ the extreme point operator $ex:\mathcal{F}\rightarrow2^{E}$ satisfies
the heritage property.

$(ii)$ for every monotone linkage function $\pi$ the function $F_{\pi}$ is quasi-concave.
\end{theorem}

\begin{proof}
Let the extreme point operator $ex$ satisfies the heritage property. To prove
that the function $F_{\pi}$ is a quasi-concave function on $\mathcal{F}$,
first note that
\begin{equation}
Z\in\mathcal{C}(X)\text{ implies }ex(Z)\subseteq X\ \text{for\ each}%
\ \text{nonempty}\ X\subseteq E\text{.} \label{two}%
\end{equation}
This statement immediately follows from the definition of a cover set.

Consider some $Z=\mathcal{C}(X\cup Y)$. Let $F_{\pi}(Z)=\underset{x\in
ex(Z)}{\min}\pi(x,Z)$. Then $F_{\pi}(Z)=\pi(x^{\ast},Z)$, where $x^{\ast}%
\in\underset{x\in ex(Z)}{\arg\min}\pi(x,Z)$. Then, by (\ref{two}), $x^{\ast
}\in X\cup Y$. \ Assume, without loss of generality, that $x^{\ast}\in X$.
Thus by the heritage property $x^{\ast}\in ex(X)$, because $x^{\ast}\in X$,
and $X\subseteq Z$, and $x^{\ast}\in ex(Z)$. Hence
\[
F_{\pi}(Z)=\pi(x^{\ast},Z)\geq\pi(x^{\ast},X)\geq\underset{x\in ex(X)}{\min
}\pi(x,X)=F_{\pi}(X)\geq\min\{F_{\pi}(X),F_{\pi}(Y)\}.
\]

Conversely, assume that the extreme point operator $ex$ does not satisfy the
heritage property, i.e., there exist $A,B\in\mathcal{F}$ such that $A\subset
B$, and there is $a\in A$ such that $B-a\in\mathcal{F}$ and $A-a\notin
\mathcal{F}$.

It is easy to check that the function%

\[
\pi(x,X)=\left\{
\begin{array}
[c]{ll}%
1, & x=a\\
2, & otherwise
\end{array}
\right.  .
\]
is monotone.

Then, $F_{\pi}(B)=1$, $F_{\pi}(A)=$ $F_{\pi}(B-a)=2$. Since $A\cup(B-a)=B $,
we have
\[
F_{\pi}(A\cup(B-a))<\min\{F_{\pi}(A),F_{\pi}(B-a)\}
\]
i.e., $F_{\pi}$ is not a quasi-concave function.
\end{proof}

Thus, if a set system $(E,\mathcal{F})$ is accessible and the operator $ex$
satisfies the heritage property, then for each set function $F$, defined on
$(E,\mathcal{F})$, one can build the quasi-concave set function $G_{F}$ that
is a upper bound of the original function $F$.

We show the corresponding property holds also for monotone linkage functions.

\begin{theorem}
\label{null_pi}Let $(E,\mathcal{F})$ be an accessible set system with the
operator $ex$ satisfying the heritage property, and let a function $F_{\pi}$
be defined as a minimum of a monotone linkage function $\pi$ by (\ref{F_def}),
then $\pi_{F}|_{\mathcal{F}}\leq\pi|_{\mathcal{F}}$, i.e., for all
$X\in\mathcal{F}$ and $x\in ex(X)\ $%
\[
\pi_{F}(x,X)\leq\pi(x,X),
\]
where $\pi_{F}$ is defined by (\ref{pi_F}).
\end{theorem}

\begin{proof}
For all $X\in\mathcal{F}$ and $x\in ex(X)$%
\[
\pi_{F}(x,X)=\underset{A\in\lbrack x,X]_{\mathcal{F}}}{\max}F(A)=F(A^{x}%
)=\underset{a\in ex(A^{\ast})}{\min}\pi(a,A^{x})\leq\pi(x,A^{x}),
\]
where $A^{x}\in\underset{A\in\lbrack x,X]_{\mathcal{F}}}{\arg\max}F(A)$.

The last inequality follows from the heritage property. Indeed, $X\supseteq
A^{x}$ and $x\in A^{x}$, then $x\in ex(X)$ implies $x\in ex(A^{x})$.

Now, from monotonicity of the function $\pi$ we have $\pi(x,A^{x})\leq
\pi(x,X)$, that finishes the proof.
\end{proof}

Consider the following example to see that the functions $\pi$ and $\pi_{F} $
can be not equal. Let $E=\{1,2\}$, $\mathcal{F}=2^{E}$.
\[
\pi(x,X)=\left\{
\begin{array}
[c]{ll}%
2, & x=2\mathit{\ }\text{and }X=\{1,2\}\\
1, & \text{otherwise}%
\end{array}
\right.
\]
then the function $F(X)=\underset{x\in ex(X)}{\min}\pi(x,X)$ is equal to $1 $
for all $X\subset E$, and then $\pi_{F}$ is equal for $1$ for each pair
$(x,X)\in E\times2^{E}$, i.e., $\pi_{F}$ $\neq\pi$.

Define more exactly the structure of the set of monotone linkage functions.

\begin{theorem}
Let $(E,\mathcal{F})$ be an accessible set system, and let $\pi_{1\text{ }}%
$and $\pi_{2}$ define (by (\ref{F_def})) the same set function $F$ on
$\mathcal{F}$. Then the function
\[
\pi=\min\{\pi_{1\text{ }},\pi_{2}\}
\]
is a monotone linkage function determines the same function $F$ on
$\mathcal{F}$.
\end{theorem}

\begin{proof}
At first, prove that $\pi$ is a monotone linkage function. Indeed, consider a
pair $X\subseteq Y$. Without loss of generality we have%
\[
\pi(x,Y)=\min\{\pi_{1}(x,Y),\pi_{2}(x,Y)\}=\pi_{1}(x,Y).
\]
Then, from monotonicity,%
\[
\pi_{1}(x,Y)\geq\pi_{1}(x,X)\geq\min\{\pi_{1}(x,X),\pi_{2}(x,X)\}=\pi(x,X)
\]

Now, denote $G(X)=\underset{x\in ex(X)}{\min}\pi(x,X)$ and prove that $G=F$.

We have%
\[
G(X)=\underset{x\in ex(X)}{\min}\pi(x,X)=\pi(x^{\ast},X)=\min\{\pi_{1}%
(x^{\ast},X),\pi_{2}(x^{\ast},X)\},
\]
where $x^{\ast}\in\underset{x\in ex(X)}{\arg\min}\pi(x,X)$. Without loss of
generality we have%
\[
G(X)=\pi_{1}(x^{\ast},X)\geq\underset{x\in ex(X)}{\min}\pi_{1}(x,X)=F(X).
\]

On the other hand,
\[
F(X)=\underset{x\in ex(X)}{\min}\pi_{1}(x,X)=\pi_{1}(x^{\#},X)\geq\pi
(x^{\#},X)\geq\underset{x\in ex(X)}{\min}\pi(x,X)=G(X).
\]

\end{proof}

Thus, the set of monotone linkage functions, defined by the set function $F$
on an accessible set system, forms a semilattice with the lattice operation
\[
\pi_{1}\wedge\pi_{2}=\min\{\pi_{1\text{ }},\pi_{2}\},
\]
where the function $\pi_{F}$ is a null of this semilattice (follows from
Theorem \ref{null_pi}).

\section{Conclusion}

Some aspects of duality between quasi-concave set functions and monotone
linkage functions were discussed for convex geometries, and more generally,
for chain systems.

Our findings may lead to efficient optimization procedures on more complex set
systems than just matroids and antimatroids.

\end{document}